\documentclass[12pt]{amsart}
\usepackage{amsmath}
\usepackage{amssymb}
\usepackage{amsfonts}

\usepackage{hyperref,url}
\usepackage{xcolor}
\hypersetup{
   colorlinks,
    linkcolor={black!50!black},
    citecolor={black!50!black},
    urlcolor={black!80!black}
}

\makeatletter
\@namedef{subjclassname@2010}{%
  \textup{2010} Mathematics Subject Classification}
\makeatother

\newcommand{\lce}{x_1+\ldots+x_k \equiv b\,(\text{mod } n)}

\newcommand{\Z}{\mathbb{Z}}
\newcommand{\N}{\mathbb{N}}

\newtheorem{thm}{Theorem}[section]

\newtheorem{lemma}[thm]{Lemma}

\usepackage{etoolbox}
\newcommand{\zerodisplayskips}{%
  \setlength{\abovedisplayskip}{0pt}%
  \setlength{\belowdisplayskip}{0pt}%
  \setlength{\abovedisplayshortskip}{0pt}%
  \setlength{\belowdisplayshortskip}{0pt}}
\appto{\normalsize}{\zerodisplayskips}
\appto{\small}{\zerodisplayskips}
\appto{\footnotesize}{\zerodisplayskips}

\title[On Solving a restricted linear congruence]{On solving a restricted linear congruence using generalized Ramanujan sums}
\author[K V Namboothiri]{K Vishnu Namboothiri}
\address{Department of Mathematics, Government Polytechnic College, Vennikkulam, Thiruvalla, Kerala - 689 544, INDIA\\Department of Collegiate Education, Government of Kerala, INDIA}
\email{kvnamboothiri@gmail.com}

\begin{document}
\baselineskip=17pt

\begin{abstract}
Consider the linear congruence equation $\lce$ for $b,n\in\Z$.  By $(a,b)_s$, we mean the largest $l^s\in\N$ which divides $a$ and $b$ simultaneously. For each $d_j|n$, define $\mathcal{C}_{j,s} = \{1\leq x\leq n^s | (x,n^s)_s = d^s_j\}$. Bibak \emph{et al.} \cite{bibak2016restricted} gave a formula using Ramanujan sums for the number of solutions of the above congruence equation with some gcd restrictions on $x_i$. We generalize their result with generalized gcd restrictions on $x_i$ by  proving that for the above linear congruence,  the number of solutions is 
$$\frac{1}{n^s}\sum\limits_{d|n}c_{d,s}(b)\prod\limits_{j=1}^{\tau(n)}\left(c_{\frac{n}{d_j},s}(\frac{n^s}{d^s})\right)^{g_j}$$ 
where $g_j = |\{x_1,\ldots, x_k\}\cap \mathcal{C}_{j,s}|$ for $j=1,\ldots \tau(n)$ and $c_{d,s}$ denote the generalized ramanujan sum defined by E. Cohen in \cite{cohen1949extension}. 
\end{abstract}



\maketitle

\section{Introduction}
The history of attempts to find general solutions of linear congruences is very old. For the general linear congruence equation
\begin{equation}
 a_1 x_1+\ldots a_k x_k\equiv b \,(\text{mod } n) \label{gen_lin_cong}
\end{equation}
D. N. Lehmer \cite{lehmer1913certain} proved the following:
\begin{thm}
 Let $a_1,\ldots,a_k,b,n\in\Z, n\geq 1$. The linear congruence equation (\ref{gen_lin_cong}) has a solution $\langle x_1,\ldots, x_n\rangle \in \Z_n^k$ if and only if $l|b$ where $l$ is the gcd of $a_1,\ldots,a_k,n$. Furthermore, if this condition is satisfied, then ther are $ln^{k-1}$ solutions.
\end{thm}

On the above type of congruence equations, if we put some restrictions on the solution set, like $gcd(x_i,n)=t_i (1\leq i \leq k)$ where $t_i$ are given positive divisors of $n$, then it is called to be a restricted linear congruence. Many authors have attempted to solve these kind of restricted congruences with varying conditions. With $a_i=1$ and restrictions $(x_i,n)=1$, Rademacher \cite{radekacher1925aufgabe}  and Brauer \cite{brauer1926losung} independenty gave a formula for the number of solutions $N_n(k,b)$ of the congruence. Their formula was
\begin{equation}
 N_n(k,b)=\frac{\varphi(n)^k}{n}\prod_{p|n,p|b}\left(1-\frac{(-1)^{k-1}}{(p-1)^{k-1}}\right) \prod_{p|n,p\nmid b}\left(1-\frac{(-1)^{k}}{(p-1)^{k}}\right)
\end{equation}
where $\varphi$ is the Euler Totient function and $p$ are all the prime divisors of $n$. An equivalent formula involving the Ramanjans sums was proved by Nicol and Vandiver \cite{nicol1954sterneck} initially, and later by E. Cohen \cite{cohen1955class}. They proved that
\begin{equation}
 N_n(k,b) = \frac{1}{n}\sum\limits_{d|n}c_d(b)\left(c_n\left(\frac{n}{d}\right)\right)^k
\end{equation}
where $c_r(n)$ denote the usual Ramanujan sum.

The restricted congruence (\ref{gen_lin_cong}) and their solutions has found interesting applications in various fields including number theory, cryptography, combinatorics, computer science etc. The special case of the problem with $b=0$ and $a_i=1$ is related to the multivariate arithmetic function defined by Liskovets (see \cite{liskovets2010multivariate}) which has many combinatorial as well as topological applications. The problem has also found use in studying universal hashing (see Bibak \emph{et al.} \cite{bibak2015almost}) which has applications in computer science.

In \cite{bibak2016restricted} Bibak \emph{et al.} considered the  linear congruence (\ref{gen_lin_cong}) taking $a_i=1$ and the restrictions $(x_i,n)=t_i$ where $t_i$ are given positive divisors of $n$. They proved the following:
\begin{thm}
 Let $b,n\in\Z, n\geq 1$, and $d_1\ldots, d_{\tau(n)}$ be the positive divisors of $n$. For $1\leq j\leq \tau(n)$, define $\mathcal{C}_j=\{1\leq x \leq n|(x,n)=d_j\}$. The number of solutions of the linear congruence $x_1+\ldots +x_k\equiv b\,(\text{mod }n)$, with $g_j= |\{x_1,\ldots, x_k\}\cap \mathcal{C}_{j}|$, $1\leq j\leq \tau(n)$, is\footnote{The formula appearing in theorem 1.1 of Bibak \emph{et al.} \cite{bibak2016restricted} seems to have mistyped $d$ in the place of $\frac{n}{d}$ in the second Ramanujan sum.} 
\begin{equation}
 \frac{1}{n}\sum\limits_{d|n}c_{d}(b)\prod\limits_{j=1}^{\tau(n)}\left(c_{\frac{n}{d_j}}\left(\frac{n}{d}\right)\right)^{g_j}
\end{equation}
 
\end{thm}
This result has been proved in some special cases by many authors, for example in \cite{cohen1955class, dixon1960finite, nicol1954sterneck, sander2013adding}. Bibak \emph{et al.} \cite{bibak2017restricted} themselves gave an alternate proof for the above result. This result was proved in \cite{bibak2016restricted} using finite Fourier transform of arithmetic functions and properties of Ramanujan sums. We generalize this result modifying the restrictions to $(x_i,n^s)_s=d_i^s$ where $d_s$ are positive divisors of $n$ and the modulus is $n^s$. We also hope that the result and proof will have some impact on attempts to solve non linear congruences as well in addition to demonstrating the diverse ways in which generalizations of Ramanujan sums can work.

\section{Notations and basic results}
For $a,b\in \Z$ with atleast one of them non zero, the generalized gcd of these numbers $(a,b)_s$ is defined to be the largest $l^s\in\N$ dividing $a$ and $b$ simultaneously. Therefore $(a,b)_1=(a,b)$, the usual gcd of two integers. $\tau(n)$ denotes the number of positive divisors of an integer $n$.

Let $c_r(n)$ denote the Ramanujan sum which is defined to be the sum of $n^{\text{th}}$ powers of primitive  $r^{\text{th}}$ roots of unity. That is, 
\begin{equation}\label{ram_sum}
 c_r(n)=\sum\limits_{j=1, (j,r)=1}^r e\left(\frac{jn}{r}\right)
\end{equation}
 For a positive integer $s$, E. Cohen \cite{cohen1949extension} generalized the Ramanujan sum defining $c_{r,s}$ as follows:
\begin{equation}
 c_{r,s}(n)=\sum\limits_{j=1,(j,r^s)_s=1}^{r^s}e\left(\frac{nj}{r^s}\right)
\end{equation}
Note that for $s=1$, this definition gives the usual Ramanujan sum defined in equation (\ref{ram_sum}). In the same paper, Cohen also gave the following formula:
\begin{equation}\label{mob_grs}
 c_{r,s}(n)=\sum\limits_{d|r, d^s|n}\mu\left(\frac{r}{d}\right)d^s
\end{equation}
where $\mu$ is the Moebius function.

For a positive integer $r$, an arithmetic function $f$ is said to be periodic with period $r$ (or $r$-periodic) if for every $m\in\Z$, $f(m+r)=f(m)$. By $e(x)$, we mean the complex exponential funcion $exp(2\pi i x)$ which has period 1.

We now have an easy, but very much useful lemma. 
\begin{lemma}
 As a function of $a$, $(a,b)_s$ is $b$-periodic.
\end{lemma}
\begin{proof}
 Note that $(a,b)_s=l^s$ is the largest $s^{\text{th}}$ power that divides $a$ and $b$ simultaneously. So $a=l^s a_1$ and $b=l^s b_1$ with $a_1$ and $b_1$ sharing no common $s^{\text{th}}$ power. Now $l^s | a+b$ and $l^s | b$. If $(a+b,b)_s=l^sl_1^s$ for some $l_1$, then $l_1^s|a+b$ and $l_1^s|b$ so that $l_1^s|a$ as well. Therfore $(a,b)_s=l^sl_1^s$ and so $l_1=1$. So $(a+b, b)_s=l^s$.
\end{proof}

Let $r,s$ be  positive integers. A function $f$ that satisfies $f(m)=f((m,r^s)_s)$ is called as an $(r,s)$-even function. This concept was introduced by McCarthy in \cite{mccarthy1960generation} and many of its properties were studied there. The above lemma says that an $(r,s)$-even function is $r^s$-periodic. 

The following appeared as Lemma (2) in \cite{cohen1950extension}:
\begin{lemma}
 If $(n,r^s)_s=l^s$, then $c_{r,s}(n)= c_{r,s}(l^s)$.
\end{lemma}

The above two lemmas combined together tell that $c_{r,s}(n)$ is $r^s$-periodic. It also follows that $c_{r,s}(-n)=c_{r,s}(n)$.

We now have one more lemma. Since we could not find a proof for this anywhere, we prove it using some elementary arguments.
\begin{lemma}
 Let $e|n$. Then $c_{e,s}$ is $(n,s)$-even. That is, $ c_{e,s}(m) = c_{e,s}\left((m,n^s)_s\right)$.
 \begin{proof}
  We use two facts; 
  \begin{enumerate}
   \item $c_{e,s}$ is $(e,s)$-even.
   \item $\left((m,n^s)_s, e^s\right)_s = l^s$ if and only if $l^s$ is the largest $s^{\text{th}}$ power dividing $(m,n^s)_s$ and $e^s$, that is if and only if $l^s$ is the largest $s^{\text{th}}$ power dividing $m,n^s$ and $e^s$. Therefore $l^s$ is the largest $s^{\text{th}}$ power dividing $m$ and $e^s$, and so $(m,e^s)_s = l^s$. The conclusion is that $((m,n^s)_s, e^s)_s = (m,e^s)_s$.
  \end{enumerate}
  
  Combining these two facts, we get
  $$c_{e,s}\left((m,n^s)_s\right) = c_{e,s}(\left((m,n^s)_s,e^s)_s\right) = c_{e,s}((m,e^s)_s) = c_{e,s}(m)$$

 \end{proof}

\end{lemma}

For an $r$-periodic arithmetic function $f(n)$, its finite Fourier transform is defined to be the function 
\begin{equation}
 \hat{f}(b) = \frac{1}{r}\sum\limits_{n=1}^r f(n) e\left(\frac{-bn}{r}\right)
\end{equation}
A Fourier representation of $f$ is given by 
\begin{equation}
 f(n)= \sum\limits_{b=1}^r \hat{f}(b) e\left(\frac{bn}{r}\right)
\end{equation}
See, for example, \cite{montgomery2006multiplicative} for a detailed study on Finite Fourier transforms.

We are now ready to state and prove our main result.
\section{The Main Theorem}
\begin{thm}
 Let $b,n\in\Z, n\geq 1$, and $d_1\ldots, d_{\tau(n)}$ be the positive divisors of $n$. For $1\leq l\leq \tau(n)$, define $\mathcal{C}_{j,s}=\{1\leq x \leq n^s|(x,n^s)_s=d^s_j\}$. The number of solutions of the linear congruence 
 \begin{equation}\label{res_lin_cong}
x_1+\ldots +x_k\equiv b\,(\text{mod }n)
 \end{equation}
with restrictions $(x_i,n^s)_s=t_i^s$, where $t_i|n$  are given for $i=1,\ldots,k$ is
\begin{equation}
\frac{1}{n^s}\sum\limits_{d|n}c_{d,s}(b)\prod\limits_{j=1}^{\tau(n)}\left(c_{\frac{n}{d_j},s}(\frac{n^s}{d^s})\right)^{g_j}
\end{equation}
 where $g_j= |\{x_1,\ldots, x_k\}\cap \mathcal{C}_{j,s}|$, $1\leq j\leq \tau(n)$.
\end{thm}

\begin{proof}
 As we have already mentioned, the proof uses the basic properties of finite Fourier transforms of $r^s$-periodic functions, the properties of generalized Ramanujan sums, and some combinatorial arguments. We follow the same approach used by Bibak \emph{et al.} in \cite{bibak2016restricted}. 
 
 Let $\hat{f}(b)$ denote the number of solutions of the linear congruence (\ref{res_lin_cong}). Therefore $\hat{f}(b)$ is the number of possible ways of writing $b$ as a sum modulo $n^s$ using $g_j$ elements in $\mathcal{C}_{j,s}$ where $j$ varies from 1 to $\tau(n)$. Note that the if $b$ is replaced with $b+n^s$ in this equation, it remains the same. So $\hat{f}(b) = \hat{f}(b+n^s)$ and therefore $\hat{f}$ is $n^s$-periodic. Let us consider the following product of exponential sums:
 
 $$\prod\limits_{j=1}^{\tau(n)}\left(\sum\limits_{x\in\mathcal{C}_{j,s}}e\left(\frac{mx}{n^s}\right)\right)^{g_j}$$
 To understand this product of sums better, put $\alpha = e(m)$. Then the product becomes $$\prod\limits_{j=1}^{\tau(n)}\left(\sum\limits_{x\in\mathcal{C}_{j,s}}\alpha^{x/n^s}\right)^{g_j}$$
 On expanding this product of sums, we get terms like $\alpha^{1/n^s}, \alpha^{2/n^s}, \ldots, \alpha^{(n^s-1)/n^s}$. Some of these powers may not even exist depending on whether sum of terms in $\mathcal{C}_{j,s}$ can be equal to that power or not. For example, $\alpha^{5/n^s}$ will not exist if the elements in various $\mathcal{C}_{j,s}$ cannot add up together to give $5\text{ modulo } n^s$. Now how many times each $\alpha^{b/n^s}$ will exist? Those many times equal to the number of possible solutions of the linear congruence with $g_j$ entries from $\mathcal{C}_{j,s}$. But this is precisely our $\hat{f}(b)$.
 
  So we get \begin{equation*}
  \sum\limits_{b=1}^{n^s} \hat{f}(b) e\left(\frac{bm}{n^s}\right) = \prod\limits_{j=1}^{\tau(n)}\left(\sum\limits_{x\in\mathcal{C}_{j,s}}e\left(\frac{mx}{n^s}\right)\right)^{g_j}
           \end{equation*}

 We now check the inner sum in this product:
 \begin{align*}
  \sum\limits_{x\in\mathcal{C}_{j,s}}e\left(\frac{mx}{n^s}\right) &= \sum\limits_{1\leq x \leq n^s, (x,n^s)_s = d_j^s}e\left(\frac{mx}{n^s}\right)\\
  &= \sum\limits_{1\leq y \leq \frac{n^s}{d_j^s}, (y,n^s)_s = 1}e\left(\frac{my}{n^s/d_j^s}\right)\\
  &= c_{\frac{n}{d_j},s}(m)
 \end{align*}
 
which gives \begin{equation*}
  \sum\limits_{b=1}^{n^s} \hat{f}(b) e\left(\frac{bm}{n^s}\right) = \prod\limits_{j=1}^{\tau(n)}\left(c_{\frac{n}{d_j},s}(m)\right)^{g_j}
           \end{equation*}

Use the fact that $\hat{f}(b)$ is $n^s$-periodic. By the finite Fourier transform theory, we get 
\begin{eqnarray*}
 \hat{f}(b) &=& \frac{1}{n^s}\sum\limits_{m=1}^{n^s}  \left[\prod\limits_{j=1}^{\tau(n)}\left(c_{\frac{n}{d_j},s}(m)\right)^{g_j}\right] e\left(\frac{-bm}{n^s}\right)\\
 & &\qquad\text{now collect the terms with same generalized gcd}\\
 &=& \frac{1}{n^s}\sum\limits_{d|n}\sum\limits_{1\leq m \leq n^s, (m,n^s)s=d^s}  \left[\prod\limits_{j=1}^{\tau(n)}\left(c_{\frac{n}{d_j},s}(m)\right)^{g_j}\right] e\left(\frac{-bm}{n^s}\right)\\
  &=& \frac{1}{n^s}\sum\limits_{d|n}\sum\limits_{1\leq m' \leq \frac{n^s}{d^s}, (m',\frac{n^s}{d^s})s=1}  \left[\prod\limits_{j=1}^{\tau(n)}\left(c_{\frac{n}{d_j},s}(m'd^s)\right)^{g_j}\right] e\left(\frac{-bm'd^s}{n^s}\right)\\
   &=& \frac{1}{n^s}\sum\limits_{d|n}\sum\limits_{1\leq m' \leq \frac{n^s}{d^s}, (m',\frac{n^s}{d^s})s=1}  e\left(\frac{-bm'}{n^s/d^s}\right)  \left[\prod\limits_{j=1}^{\tau(n)}\left(c_{\frac{n}{d_j},s}((m'd^s, n^s)_s)\right)^{g_j}\right]\\
  &=& \frac{1}{n^s}\sum\limits_{d|n}\sum\limits_{1\leq m' \leq \frac{n^s}{d^s}, (m',\frac{n^s}{d^s})s=1}  e\left(\frac{-bm'}{n^s/d^s}\right)  \left[\prod\limits_{j=1}^{\tau(n)}\left(c_{\frac{n}{d_j},s}(d^s)\right)^{g_j}\right]\\
  &=& \frac{1}{n^s}\sum\limits_{d|n}c_{\frac{n}{d},s}(b) \left[\prod\limits_{j=1}^{\tau(n)}\left(c_{\frac{n}{d_j},s}(d^s)\right)^{g_j}\right]\\
  &=& \frac{1}{n^s}\sum\limits_{d|n}c_{d,s}(b) \left[\prod\limits_{j=1}^{\tau(n)}\left(c_{\frac{n}{d_j},s}(\frac{n^s}{d^s})\right)^{g_j}\right]
\end{eqnarray*}
\end{proof}

We will give a small example to demonstrate the result: Consider the linear congruence $x_1+x_2 \equiv 5\, (\text{mod } 16)$. Here $n^s=4^2, b=5, k=2$.
$\{d_j^2\} = \{1,4,16\}$ and so \\
$\mathcal{C}_{1,2} = \{1,2,3,5,6,7,9,10,11,13,14,15\}$\\
$\mathcal{C}_{2,2} = \{4,8,12\}$\\
$\mathcal{C}_{3,2} = \{16\}$

Suppose that we want to find solutions with the restrictions $(x_1,16)_2=1$ and $(x_2,16)_2 = 4$. In this case $g_1 = 1, g_2=1, g_3=0$. By simple observation, we get the number of solutions to be 3 which are $\langle 1,4\rangle, \langle9,12\rangle, \langle13,8\rangle$. Now according to our formula, the computation is the follwing:
\begin{eqnarray*}
 \sum\limits_{d|4}c_{d,2}(5)\left[\prod\limits_{j=1}^3 c_{\frac{4}{d_j},2}\left(\frac{16}{d^2}\right)^{g_j}\right] = & &c_{1,2}(5)\times \left[\prod\limits_{j=1}^3 c_{\frac{4}{d_j},2}\left(\frac{16}{1^2}\right)^{1}\right]\\
 & &+ c_{2,2}(5)\times \left[\prod\limits_{j=1}^3 c_{\frac{4}{d_j},2}\left(\frac{16}{2^2}\right)^{1}\right]\\
 & &+ c_{4,2}(5)\times \left[\prod\limits_{j=1}^3 c_{\frac{4}{d_j},2}\left(\frac{16}{4^2}\right)^{1}\right] \\
 =& & c_{1,2}(5)\times \left[ c_{4,2}\left(16\right)\times c_{2,2}\left(16\right) \times c_{1,2}\left(16\right)\right]\\
 &+& c_{2,2}(5)\times \left[ c_{4,2}\left(4\right)\times c_{2,2}\left(4\right) \times c_{1,2}\left(4\right)\right]\\
 &+& c_{4,2}(5)\times \left[ c_{4,2}\left(1\right)\times c_{2,2}\left(1\right) \times c_{1,2}\left(1\right)\right]\\
 =& & 1\times 12.3.1+(-1)\times (-4).3.1+0\times 0.(-1).1\\
 =& & 48
\end{eqnarray*}
which on division by 16 gives 3 as the number of solutions. We have used identity (\ref{mob_grs}) to evaluate $c_{r,s}$ at various values.

\bibliography{nt} 
\bibliographystyle{plain}

\end{document}